\documentclass[oneside]{amsart}
\usepackage{amsmath}
\usepackage{amsfonts}
\usepackage{amsthm}
\usepackage{hyperref}
\usepackage{amscd}
\usepackage{color}
\usepackage{enumerate}
\usepackage{subcaption}
\DeclareCaptionSubType[Alph]{figure}

\newcommand{\bF}{\mathbf{F}}

\newcommand{\KK}{\mathbb{K}}
\newcommand{\PP}{\mathbb{P}}
\newcommand{\CC}{\mathbb{C}}
\newcommand{\A}{\mathbb{A}}

\newcommand{\NN}{\mathbb{N}}

\newcommand{\Perm}{\mathbf{perm}}
\newcommand{\Det}{\mathbf{det}}
\newcommand{\cX}{\mathcal{X}}

\DeclareMathOperator{\pr}{pr}
\newcommand{\bpr}{\underline{\pr}}
\DeclareMathOperator{\tr}{tr}
\let\wr\relax
\DeclareMathOperator{\wr}{wr}

\DeclareMathOperator{\spec}{Spec}

\newtheorem{thm}{Theorem}[section]
\newtheorem{prop}[thm]{Proposition}

\theoremstyle{definition}

\newtheorem{rem}[thm]{Remark}

\title{Product Ranks of the $3\times 3$ Determinant and Permanent}

\date{March 16, 2015}

\author{Nathan Ilten}
\address{Department of Mathematics, Simon Fraser University,
8888 University Drive, Burnaby BC V5A1S6, Canada}
\email{nilten@sfu.ca}

\author{Zach Teitler}
\address{Department of Mathematics,
Boise State University,
1910 University Drive
Boise, ID 83725-1555, USA}
\email{zteitler@boisestate.edu}


\begin{document}

\begin{abstract}
We show that the product rank of the $3 \times 3$ determinant $\Det_3$ is $5$,
and the product rank of the $3 \times 3$ permanent $\Perm_3$ is $4$.
As a corollary, we obtain that the tensor rank of $\Det_3$ is $5$ and the tensor rank of $\Perm_3$ is $4$.
We show moreover that the border product rank of $\Perm_n$ is larger than $n$ for any $n\geq 3$.
\end{abstract}
\maketitle
\section*{Introduction}

Let $A = (a_{ij})$ be an $n \times n$ matrix.
Recall that the permanent of $A$, denoted $\Perm(A)$, is given by
\[
  \Perm(A) = \sum_{\sigma \in S_n} a_{1\sigma(1)} \dotsm a_{n\sigma(n)},
\]
the sum over the symmetric group $S_n$ of permutations of $\{1,\dotsc,n\}$.
We write $\Perm_n = \Perm((x_{ij}))$ for the permanent of the $n \times n$ generic matrix, that is,
a matrix whose entries are independent variables.
The definition writes $\Perm_n$ as a sum of $n!$ terms which are products of linear forms, in fact variables.
Allowing terms involving products of linear forms other than variables allows for more efficient representations.
For example Ryser's identity \cite{MR0150048} gives
\[
  \Perm_n = \sum_{S \subseteq \{1,\dotsc,n\}} (-1)^{n - |S|} \prod_{i=1}^n \sum_{j \in S} x_{ij} .
\]
This uses $2^n-1$ terms.
Even better, Glynn's identity \cite{MR2673027} gives
\[
  \Perm_n = \sum_{\substack{\epsilon \in \{\pm1\}^n \\ \epsilon_1 = 1}} \prod_{i=1}^n \sum_{j=1}^n \epsilon_i \epsilon_j x_{ij} .
\]
This uses $2^{n-1}$ terms.
For example, $\Perm_3$ can be written as a sum of $4$ terms which are products of linear forms.
Explicitly,
\[
\begin{split}
  \Perm_3 &= (x_{11} + x_{12} + x_{13})(x_{21} + x_{22} + x_{23})(x_{31} + x_{32} + x_{33}) \\
    &- (x_{11} + x_{12} - x_{13})(x_{21} + x_{22} - x_{23})(x_{31} + x_{32} - x_{33}) \\
    &- (x_{11} - x_{12} + x_{13})(x_{21} - x_{22} + x_{23})(x_{31} - x_{32} + x_{33}) \\
    &+ (x_{11} - x_{12} - x_{13})(x_{21} - x_{22} - x_{23})(x_{31} - x_{32} - x_{33}) .
\end{split}
\]
We will show that it is not possible to write $\Perm_3$ as a sum of $3$ or fewer such terms.
In fact, we will show that it is not possible to write $\Perm_3$ as a limit of cubic polynomials using $3$ or fewer such terms.

Similarly we write $\Det_n$ for the determinant of an $n \times n$ generic matrix.
The Laplace expansion writes $\Det_n$ as a sum of $n!$ monomials.
In particular $\Det_3$ is a sum of $6$ monomials; until recently it was not clear whether $\Det_3$ could be written as a sum
of products of linear forms using $5$ or fewer terms.
However Derksen recently found such an expression \cite[\textsection 8]{Derksen:2013sf}:
\[
\begin{split}
 \Det_3 = \frac{1}{2} \Big( & (x_{13}+x_{12})(x_{21}-x_{22})(x_{31}+x_{32}) \\
   & + (x_{11}+x_{12})(x_{22}-x_{23})(x_{32}+x_{33}) \\
   & + 2 x_{12}(x_{23}-x_{21})(x_{33}+x_{31}) \\
   & + (x_{13}-x_{12})(x_{22}+x_{21})(x_{32}-x_{31}) \\
   & + (x_{11}-x_{12})(x_{23}+x_{22})(x_{33}-x_{32})
   \Big) .
\end{split}
\]
In hindsight it should have been clear that such an expression must exist. Indeed, over e.g.\ $\CC$,
$\Det_3$ can be regarded as a tensor in $\CC^3 \otimes \CC^3 \otimes \CC^3$, and
it is known that all such tensors have rank at most $5$ \cite{MR3089693}. As we shall see, this implies an expression involving at most $5$ products of linear forms.
Nevertheless, this does not seem to have been noticed previously.

In any case $\Det_3$ can be written as a sum of $5$ products of linear forms.
We show that  is not possible to write $\Det_3$ as a sum of $4$ or fewer such terms.

For both the permanent and determinant, the key ingredient in our proofs is an analysis of certain Fano schemes parametrizing linear subspaces
contained in the hypersurfaces $\Perm_3 = 0$ and $\Det_3 = 0$. We hope that our techniques may be employed to attack other similar problems in tensor rank and algebraic complexity theory.

\subsection*{Acknowledgements} We thank J.M.~Landsberg for helpful comments.

\section{Product Rank}

\subsection{Basic notions}
Throughout we work over some fixed field $\KK$ of characteristic zero.
Recall that the \emph{rank} or \emph{tensor rank} of a tensor $T \in V_1 \otimes \dotsb \otimes V_k$
is the least number of terms $r$ in an expression
\[
  T = \sum_{i = 1}^r v_{1i} \otimes \dotsb \otimes v_{ki}.
\]
We denote the tensor rank by $\tr(T)$.
Recall also that the \emph{Waring rank} of a homogeneous form $F$ of degree $d$
is the least number of terms $r$ in an expression
\[
  F = \sum_{i = 1}^r c_i l_i^d,
\]
where each $l_i$ is a homogeneous linear form and each $c_i \in \KK$.
We denote the Waring rank $\wr(F)$.
For overviews of tensor rank and Waring rank, including applications and history,
 we refer to \cite{MR2535056}, \cite{MR2447451}, \cite{MR2865915}.

Here we are concerned with the product rank, also called split rank or Chow rank,
see for example \cite{MR3166068}.
For a homogeneous form $F$ of degree $d$, the \emph{product rank}, denoted $\pr(F)$,
is the least number of terms $r$ in an expression
\[
  F = \sum_{i = 1}^r \prod_{j = 1}^d l_{ij},
\]
each $l_{ij}$ a homogeneous linear form. This is related to the minimum size of any homogeneous $\Sigma\Pi\Sigma$-circuit computing $F$, see \cite[\S 8]{Landsberg:2014kq}  for details.

The \emph{border product rank} $\bpr(F)$ is the least $r$ such that
$F$ is a limit of forms of product rank $r$: \[\lim_{t\to 0} F_t = F\] for some forms $F_t$ with $\pr(F_t) = r$ for $t \neq 0$.
Taking the constant family $F_t = F$ shows $\bpr(F) \leq \pr(F)$.

Note that $\bpr(F) = r$ if and only if $F$ lies in the closure of the locus of forms of product rank $r$,
but not in the closure of the forms of product rank $r-1$.
The closure of the forms of product rank $r$ is exactly the $r$th secant variety
of the variety of completely decomposable forms, that is, forms which decompose as products of linear forms.
The latter is also called the split variety or the Chow variety of zero-cycles of degree $d$ in (the dual space) $\PP^n$.
So $\bpr(F) = r$ if $F$ lies on the $r$th, but not the $(r-1)$st, secant variety of the Chow variety.
Furthermore, $\pr(F) = r$ if $F$ lies in the span of some $r$ distinct points on the Chow variety.
See \cite{MR3166068} for details.

\subsection{Waring rank and product rank}
Evidently $\pr(F) \leq \wr(F)$.
On the other hand, the expression
\[
  l_1 \dotsm l_d = \frac{1}{2^{d-1} d!} \sum_{\substack{\epsilon \in \{\pm1\}^d \\ \epsilon_1 = 1}}
    \big( \prod \epsilon_i \big) \big( \sum \epsilon_i l_i \big)^d
\]
means that 
\[
  \wr(l_1 \dotsm l_d) \leq 2^{d-1}.
\]
In fact, it is equal when the $l_i$ are linearly independent \cite{MR2842085}. In any case, we thus have
\[
  \wr(F) \leq 2^{d-1} \pr(F).
\]
For our purposes, this means that a lower bound for Waring rank implies a lower bound for product rank.
And in fact, lower bounds for the Waring ranks of determinants and permanents have been found
by Shafiei \cite{Shafiei:ud} and Derksen and the second author \cite{Derksen:2014hb}:
\[
  \wr(\Perm_n) \geq \frac{1}{2} \binom{2n}{n}, \qquad \wr(\Det_n) \geq \binom{2n}{n} - \binom{2n-2}{n-1} .
\]
For $n = 3$, this is $\wr(\Perm_3) \geq 10$ and $\wr(\Det_3) \geq 14$.
Hence, $\pr(\Perm_3) \geq 3$ and $\pr(\Det_3) \geq 4$.
On the other hand, the Glynn and Derksen identities above show $\pr(\Perm_3) \leq 4$ and $\pr(\Det_3) \leq 5$. We will show that one cannot do better than this, that is, $\bpr(\Perm_3)=\pr(\Perm_3) = 4$ and $\pr(\Det_3) = 5$.

\subsection{Tensor rank and product rank}
There is also a connection between tensor rank and product rank.
Tensors in $V_1 \otimes \dotsb \otimes V_d$ can be naturally identified with multihomogeneous forms
of multidegree $(1,\dotsc,1)$ on the product space $V_1 \times \dotsb \times V_d$.
Explicitly let each $V_i$ have a basis $x_{i1},\dotsc,x_{in_i}$
and consider polynomials in the $x_{ij}$ with multigrading in $\NN^d$
where each $x_{ij}$ has multidegree $e_i$, the $i$th basis vector of $\NN^d$.
Then each simple (basis) tensor $x_{1j_1} \otimes \dotsb \otimes x_{dj_d}$ is multihomogeneous
of multidegree $(1,\dotsc,1)$ and in fact tensors correspond precisely to multihomogeneous forms
of this multidegree.

Arbitrary simple tensors $v_1 \otimes \dotsb \otimes v_d$ correspond to products of linear forms
$l_1 \dotsm l_d$ with each $l_i$ multihomogeneous of multidegree $e_i$.
Hence $\tr(T) \geq \pr(T)$, where we slightly abuse notation by writing $T$ for both a tensor
and the corresponding multihomogeneous polynomial.

In particular our results will show $\tr(\Perm_3) \geq 4$ and $\tr(\Det_3) \geq 5$.
On the other hand the Glynn and Derksen identities involve sums of products of linear forms
which happen to be multihomogeneous (in the rows of the $3 \times 3$ matrix),
hence correspond to tensor decompositions.
So $\tr(\Perm_3) \leq 4$ and $\tr(\Det_3) \leq 5$.
In fact, Derksen gave his identity originally in tensor form.

\section{The Permanent}

\begin{thm}\label{thm:perm}
Let $n>2$. Then we have $\bpr (\Perm_n)> n$.
\end{thm}
\begin{proof}
Suppose that $\bpr(\Perm_n)\leq n$. Then there exists a smooth curve $C$ with special point $0\in C$ and an irreducible family $\cX\subset \KK^{n^2}\times C$
 with $\pi:\cX\to C$ the projection such that 
\[
\pi^{-1}(0)=\cX_0=V(\Perm_n)
\]
and for $c\neq 0$, $\pi^{-1}(c)=\cX_c$ is the vanishing locus of 
\[
F=\sum_{i=1}^n \prod_{j=1}^n x_{ij}
\]
in $\KK^{n^2}$ up to a homogeneous linear change of coordinates.

Let $\bF(\cX_c)$ denote the Fano scheme parametrizing $k=n(n-1)$-dimensional linear spaces contained in $\cX_c\subset\KK^{n^2}$; see \cite{eisenbud:00a} for details on Fano schemes. 
Then $\bF(\cX_0)$ consists of exactly $2n$ isolated points, see
\cite[Cor.\ 5.6]{Chan:2014la}.
%\cite[Cor. 5.6]{ilten:fanodet}.
The corresponding $k$-planes arise exactly by zeroing out one row or one column of an $n\times n$ matrix.
In any case, $\bF(\cX_0)$ is zero-dimensional, of degree $2n$.

On the other hand, for $c\neq 0$, $\bF(\cX_c)$ contains at least $n^n$ points.\footnote{In fact, a straightforward calculation shows that there are exactly $n^n$ points in this Fano scheme.} Indeed, the $k$-plane $V(x_{1j_1},\ldots,x_{nj_n})$ is clearly contained in $V(F)$ for any $1\leq j_1,\ldots,j_n\leq n$.
But this is impossible. Indeed, $\dim \bF(\cX_c)\leq \dim \bF(\cX_0)$ by semicontinuity of fiber dimension of proper morphisms \cite[\S 13.1.5]{EGA4}, since these Fano schemes appear as fibers in the proper map from the relative Fano scheme of $\cX/C$ to $C$.
Hence, $\dim \bF(\cX_c)=0$, so $\deg \bF(\cX_c)\geq n^n >\deg \bF(\cX_0)$, which contradicts e.g.~\cite[Proposition 4.2]{ilten:14a}.
%
%
%
%Let $k=n(n-1)$.
%We will use the fact that there are exactly $2n$ planes of dimension $k$ contained in $V(\Perm_n)\subset \KK^{n^2}$, see \cite[Cor. 5.6]{ilten:fanodet}. These planes arise exactly by zeroing out one row or one column of an $n\times n$ matrix, and are maximal in the sense that they cannot be extended within $V(\Perm_n)$.
%
%We suppose that $\pr (\Perm_n)\leq n$. Then we can write
%$$
%\Perm_n=\sum_{i=1}^n\prod_{j=1}^n l_{ij}
%$$
%for $n^2$ linear forms $l_{ij}\in \KK[x_{ij}\ |\ i,j\leq n]$. Now, if there is a linear dependence among the $l_{ij}$, then we can find some non-zero $v\in\KK^n$ such that $l_{ij}(v)=0$ for all $i,j$. Hence, $v$ is contained in every maximal plane of $V(\Perm_n)\subset \KK^{n^2}$. However, from the above description of the $k$-planes  contained in $V(\Perm_n)$, this cannot be.
%
%Hence, we must have that the $l_{ij}$ are linearly independent, in which case we can take them to be coordinates on $\KK^{n^2}$. Now, the $k$-plane $V(l_{1j_1},\ldots,l_{nj_n})$ is clearly contained in $V(\sum_{i=1}^n\prod_{j=1}^n l_{ij})$ for any $1\leq j_1,\ldots,j_n\leq n$.\footnote{A straightforward calculation shows that these are the only $k$-planes contained in this hypersurface.} Hence, $V(\Perm_n)$ contains at least $n^n$ $k$-planes, contradicting the above.
%
%Thus, $\pr (\Perm_n)> n$.
\end{proof}

\begin{rem}
In the case $n=3$, it follows that \[\tr(\Perm_3)=\pr (\Perm_3)=\bpr(\Perm_3)=4,\] since Glynn's identity gives an explicit expression showing $\pr(\Perm_3) \leq \tr(\Perm_3)\leq 4$.
On the other hand, for $n>3$, the resulting bound $\pr(\Perm_n)>n$ is weaker than the bound
$\pr(\Perm_n) \geq \frac{1}{2^n}\binom{2n}{n} \approx \frac{2^n}{\sqrt{n\pi}}$
obtained from Shafiei's bound for $\wr(\Perm_n)$.
However, our bound on $\bpr(\Perm_n)$ is the best bound we know.
\end{rem}

\section{The Determinant}

\begin{thm}\label{thm:det}
We have $\tr(\Det_3) = \pr(\Det_3) = 5$.
\end{thm}
Before beginning the proof, we need a result about a special Fano scheme. Let
$$
X=V(y_1y_2y_3+\dotsb+y_{10}y_{11}y_{12})\subset \KK^{12}=\spec \KK[y_1,\ldots,y_{12}],
$$
and let $\bF(X)$ be the Fano scheme parametrizing $6$-dimensional linear spaces of $X$. 
Let $G$ be the subgroup of $S_{12}$ acting by permuting coordinates which maps $X$ to itself.

\begin{prop}\label{prop:fano}
Consider any irreducible component $Z$ of $\bF(X)$ such that the $6$-planes parametrized by $Z$ do not all lie in a coordinate hyperplane of $\KK^{12}$. Then $Z$ is $4$-dimensional, and it can be covered by affine spaces $\A^4=\spec \KK[p,q,r,s]$. The corresponding parametrization of $6$-planes is given by the rowspan of 
$$
\left(\begin{array}{c c c c c c c c c c c c}
1 &  &  & p &  &  &  &  &  &  &  & \\
 & 1 &  &  & q &  &  &  &  &  &  & \\
 &  & -pq &  &  & 1 &  &  &  &  &  & \\
 &  &  &  &  &  & 1 &  &  & r &  & \\
 &  &  &  &  &  &  & 1 &  &  & s &  \\
 &  &  &  &  &  &  &  & -rs &  &  & 1 \\
\end{array}\right)
$$
up to some permutation in $G$.

\end{prop}
\begin{proof}
Consider the torus $T\subset (\KK^*)^{12}$ defined by the equations
$$
y_1y_2y_3=y_4y_5y_6=y_7y_8y_9=y_{10}y_{11}y_{12};
$$
$X$ is clearly fixed under the action of $T$. This torus $T$ also acts on $\bF(X)$, and, up to permutations by $G$,  has exactly the fixed points given by the spans of $e_5,e_6,e_8,e_9,e_{11},e_{12}$ and $e_3,e_6,e_8,e_9,e_{11},e_{12}$, respectively. Here, the $e_i$ are the standard basis of $\KK^{12}$.

Now, since every irreducible component of a projective scheme with a torus action contains a toric fixed point,
every irreducible component $Z$ of $\bF(X)$ must intersect one of the two Pl\"ucker charts containing the above two fixed points, up to permutations by $G$.
These two corresponding charts of the Grassmannian $G(6,12)$ are parametrized by the rowspans of the matrices
$$
A=\bgroup\makeatletter\c@MaxMatrixCols=12\makeatother\begin{pmatrix}{a}_{11}&
{a}_{12}&
{a}_{13}&
{a}_{14}&
1&
0&
{a}_{15}&
0&
0&
{a}_{16}&
0&
0\\
{a}_{21}&
{a}_{22}&
{a}_{23}&
{a}_{24}&
0&
1&
{a}_{25}&
0&
0&
{a}_{26}&
0&
0\\
{a}_{31}&
{a}_{32}&
{a}_{33}&
{a}_{34}&
0&
0&
{a}_{35}&
1&
0&
{a}_{36}&
0&
0\\
{a}_{41}&
{a}_{42}&
{a}_{43}&
{a}_{44}&
0&
0&
{a}_{45}&
0&
1&
{a}_{46}&
0&
0\\
{a}_{51}&
{a}_{52}&
{a}_{53}&
{a}_{54}&
0&
0&
{a}_{55}&
0&
0&
{a}_{56}&
1&
0\\
{a}_{61}&
{a}_{62}&
{a}_{63}&
{a}_{64}&
0&
0&
{a}_{65}&
0&
0&
{a}_{66}&
0&
1\\
\end{pmatrix}\egroup
$$
$$
B=\bgroup\makeatletter\c@MaxMatrixCols=12\makeatother\begin{pmatrix}{b}_{11}&
{b}_{12}&
1&
{b}_{13}&
{b}_{14}&
0&
{b}_{15}&
0&
0&
{b}_{16}&
0&
0\\
{b}_{21}&
{b}_{22}&
0&
{b}_{23}&
{b}_{24}&
1&
{b}_{25}&
0&
0&
{b}_{26}&
0&
0\\
{b}_{31}&
{b}_{32}&
0&
{b}_{33}&
{b}_{34}&
0&
{b}_{35}&
1&
0&
{b}_{36}&
0&
0\\
{b}_{41}&
{b}_{42}&
0&
{b}_{43}&
{b}_{44}&
0&
{b}_{45}&
0&
1&
{b}_{46}&
0&
0\\
{b}_{51}&
{b}_{52}&
0&
{b}_{53}&
{b}_{54}&
0&
{b}_{55}&
0&
0&
{b}_{56}&
1&
0\\
{b}_{61}&
{b}_{62}&
0&
{b}_{63}&
{b}_{64}&
0&
{b}_{65}&
0&
0&
{b}_{66}&
0&
1\\
\end{pmatrix}\egroup.
$$
Imposing the condition that these $6$-planes be contained in $X$ leads to the ideals $I_A\subset \KK[a_{ij}]$ and $I_B\subset \KK[b_{ij}]$ for the Pl\"ucker charts of $\bF(X)$.
We are interested in the irreducible decompositions of $V(I_A)$ and $V(I_B)$, in other words, in minimal primes of $I_A$ and $I_B$. Furthermore, since we only care about components parametrizing $6$-planes not lying in a hyperplane of $\KK^{12}$, we may discard any minimal primes containing all $a_{ij}$ or $b_{ij}$ for some fixed $j$.

Now, it is easy to see that $a_{i1}a_{i2}a_{i3}\in I_A$ for $i=1,\ldots,6$, and likewise, 
$b_{11}b_{12}$ and $b_{23}b_{24}$ are in $I_B$. Using the action of $G$, we may thus assume that for any minimal prime $P_A$ of $I_A$, $a_{11},a_{63}\in P_A$ and for any minimal prime $P_B$ of $I_B$,  $b_{11},b_{23}\in b_A$. We now proceed as follows starting with the ideal $J=I_A+\langle a_{11},a_{63}\rangle$ or $J=I_B+\langle b_{11},b_{23}\rangle$:
\begin{enumerate}
\item Find the minimal primes $\{P_1,\ldots,P_m\}$ of the ideal $J'$ generated by the monomials among a set of minimal generators of $J$;
\item Discard those $P_k$ such that $J+P_k$ contains all $a_{ij}$ or $b_{ij}$ for some fixed $j$;
\item Return to the first step, replacing $J$ by $J+P_k$ for each remaining prime $P_k$.
\end{enumerate}

We continue this process until it stabilizes, that is, among the $J+P_k$ we have no new ideals. Doing this calculation with {\tt Macaulay2} \cite{M2} (see Appendix \ref{sec:code} for code) takes less than 20 seconds on a modern computer. In the case of $I_A$, we are left with no ideals, that is, all minimal primes of $I_A$ contain all $a_{ij}$ for some fixed $j$. In the case of $I_B$, we are left with $8$ ideals, corresponding to components whose parametrization is exactly of the form postulated by the proposition. Each of these components is toric (with respect to a quotient of $T$) and projective, hence admits an invariant affine cover, each of whose charts contains a $T$-fixed point. The claim now follows.
\end{proof}

\begin{proof}[Proof of Theorem \ref{thm:det}]
We will use the fact that $6$-planes contained in $V(\Det_3)\subset \KK^{9}$ are parametrized by two copies of $\PP^2$, see \cite{Chan:2014la}. Furthermore, every point of $V(\Det_3)$ is contained in such a plane.

To begin with, we have that $\pr(\Det_3)>3$, as follows from the lower bound on the Waring rank of $\Det_3$. 
Let us assume that $\pr(\Det_3)= 4$.
We now consider the hypersurface $X$ from Proposition \ref{prop:fano}.
Our assumption implies that there is a $9$-dimensional linear subspace $L\subset \KK^{12}$ such that $V(\Det_3)=X\cap L$. Furthermore, there must be a component $Z$ of $\bF(X)$ containing a copy of $\PP^2$ such that the $6$-planes parametrized by this $\PP^2$ are all contained in $L$ (and hence in $V(\Det_3)$). Since these $6$-planes sweep out $V(\Det_3)$, the planes parametrized by the component $Z$ must not all be contained in a coordinate hyperplane $V(y_i)$ of $\KK^{12}$, otherwise
$L$ would be also be contained in $V(y_i)$. But in that case, we can clearly write $\Det_3$ as a sum of three products of linear forms, contradicting the assumption that $\pr(\Det_3)>3$.

We can now apply Proposition \ref{prop:fano} to the component $Z$.
On a local chart, the subvariety $\PP^2\subset Z$ must be cut out by setting either $p,q$ constant or $r,s$ constant.
Indeed, suppose that $p$ and $r$ are non-constant.
Each of $q$ and $s$ is either non-constant or constant but nonzero,
for if $q=0$ or $s=0$ is constant on the $\PP^2$ then the $6$-planes parametrized by the $\PP^2$ are contained
in a coordinate hyperplane in $\KK^{12}$.
Then $pq$ and $rs$ are also non-constant,
so the corresponding $6$-planes span at least a $10$-dimensional subspace of $\KK^{12}$ and hence cannot all be contained in $L$.

Thus, making use of symmetry, we may assume that $p,q$ are constant.
But if this is the case,  then $L$ must be cut out by 
\begin{align*}
y_3=-pq y_6, \quad y_4=py_1,\quad y_5=qy_2.
\end{align*}
Hence, up to homogeneous linear change of coordinates,
$X\cap L=V(\Det_3)\subset\KK^9$ is cut out by
$$
y_7y_8y_9+y_{11}y_{12}y_{13}
$$
which contradicts $\pr(\Det_3)>3$.

We conclude that $\pr(\Det_3)>4$.
Combining this with Derksen's identity shows that $\tr(\Det_3) = \pr(\Det_3)=5$.
\end{proof}

\appendix
\newpage
\section{Code for {\tt Macaulay2}}\label{sec:code}
\begin{verbatim}
R=QQ[x_1..x_12]
f=x_1*x_2*x_3+x_4*x_5*x_6+x_7*x_8*x_9+x_10*x_11*x_12
S=QQ[a_(1,1)..a_(6,6)]
N=transpose genericMatrix(S,6,6)
O=id_(S^6) 
M_A=N_{0,1,2,3}|O_{0,1}|N_{4}|O_{2,3}|N_{5}|O_{4,5}
M_B=N_{0,1}|O_{0}|N_{2,3}|O_{1}|N_{4}|O_{2,3}|N_{5}|O_{4,5}
T=S[s_1..s_6]
p_A=map(T,R,(vars T)* sub(M_A,T))
p_B=map(T,R,(vars T)* sub(M_B,T))
-- These are the ideals for the two charts:
I_A=ideal sub((coefficients p_A(f))_1,S)
I_B=ideal sub((coefficients p_B(f))_1,S)

--Detects if a component only contains linear spaces contained 
--in a coordinate hyperplane
lowRank=J->(genlist:=flatten entries mingens J;
    any(toList (1..6),i->(
        all(toList (1..6),j->member(a_(j,i),genlist)))))

--Deletes multiple occurrences of an ideal in a list
uniqueIdealList=L->(outlist:={};
    scan(L,i->(if not any(outlist,j->j==i) then outlist=outlist|{i}));
    outlist)     

--Writes an ideal as an intersection of multiple ideals, up to radical
partialDecomposition=J->(genlist:=flatten entries mingens J;
    monlist:=select(genlist,i->size i==1);
    dl:=decompose monomialIdeal ideal monlist;
    select(apply(dl,i->i+J),i->not lowRank i))

--verify that a_(i,1)*a_(i,2)*a_(i,3) are in I_A:
transpose mingens I_A
--by symmetry, can assume a_(1,1)=0, a_(6,3)=0
L1=partialDecomposition (I_A+ideal {a_(1,1),a_(6,3)});
L2=uniqueIdealList flatten (L1/partialDecomposition);
# flatten (L2/partialDecomposition)
--everything has low rank!

--verify that a_(1,1)*a_(1,2), and a_(2,3)*a_(2,4) are in I_B:
transpose mingens I_B
--by symmetry, can assume a_(1,1)=0,  a_(2,3)=0
L1=partialDecomposition (I_B+ideal {a_(1,1),a_(2,3)});
L2=uniqueIdealList flatten (L1/partialDecomposition);
L3=uniqueIdealList flatten (L2/partialDecomposition);
scan(#L3,i->(print i;print transpose mingens L3_i))
--everything has low rank or desired form!
\end{verbatim}
\bibliographystyle{amsalpha}
\renewcommand{\MR}[1]{\relax}
\bibliography{pr}
\end{document}